\documentclass[a4paper,12pt]{amsart}
\usepackage[top=1.5in, bottom=1in, left=0.9in, right=0.9in]{geometry}
\usepackage[utf8]{inputenc}
\usepackage[all]{xy}
\usepackage{color}
\usepackage{amsmath, mathtools}
\usepackage{hyperref}
\hypersetup{colorlinks=true}
\usepackage{amssymb}
\usepackage{amsfonts}
\usepackage{graphicx}
\usepackage{mathrsfs}
\usepackage{amscd}
\usepackage[all]{xy}
\usepackage{hyperref}
\usepackage{amsthm}
\usepackage{cleveref}
\usepackage{verbatim}
\usepackage{amscd}
\usepackage{mathtools}
\usepackage{comment}
\usepackage{xfrac}
\newtheorem{thm}{\bf Theorem}[section]
\newtheorem{theoremx}{\bf Theorem}

\newtheorem{questionx}[theoremx]{\bf Question}

\newtheorem{prop}[thm]{\bf Proposition}
\newtheorem{lemma}[thm]{\bf Lemma}
\newtheorem{cor}[thm]{\bf Corollary}
\theoremstyle{definition}
\newtheorem{definition}[thm]{\bf Definition}
\theoremstyle{remark}

\newtheorem{example}[thm]{\bf Example}

\numberwithin{equation}{section}

\DeclareMathOperator{\Spec}{{Spec}}

\DeclareMathOperator{\Ass}{{Ass}}

\DeclareMathOperator{\Char}{{char}}
\DeclareMathOperator{\Min}{{Min}}

\DeclareMathOperator{\Sing}{{{Sing}}}

\DeclareMathOperator{\Mat}{{Mat}}

\def\fm{\mathfrak{m}}
\def\fn{\mathfrak{n}}

\def\fp{\mathfrak{p}}

\def \KK{\mathbb K}

\def \ZZ{\mathbb Z}
\def \NN{\mathbb N}
\def \VV{\mathbb V}
\def \II{\mathbb I}
\def \A{\mathcal A}

\newcommand{\bigzero}{\mbox{\normalfont\Large\bfseries 0}}

\begin{document}
\title[Symbolic Powers of Derksen Ideals]{Symbolic Powers of Derksen Ideals} 
\author[Sandra Sandoval-Gómez]{Sandra Sandoval-Gómez}
\address{University of Notre Dame\\ Department of Mathematics, 255 Hurley, Notre Dame, IN, 46556, USA.}
\email{ssandov3@nd.edu}
\subjclass[2010]{Primary 	 ; Secondary  .}
\begin{abstract}
Given that symbolic and ordinary powers of an ideal do not always coincide, we look for conditions on the ideal such that the equality holds for every natural number. This paper focuses on studying the equality for Derksen ideals defined by finite groups acting linearly on a polynomial ring. 
\end{abstract}

\maketitle

\section{Introduction}

Given a polynomial ring in $d$ variables $R=\KK[x_1, \ldots, x_d]$ and a finite group $G$ which acts on $R$, the Derksen ideal of $S = \KK[x_1, \ldots, x_d, y_1, \ldots, y_d]$ is defined by
$$
I_G = \bigcap_{g \in G} (y_1 - g(x_1), \ldots , y_d - g(x_d)).
$$
The generators of the Derksen ideal give the generators of the invariant ring $R^G$ \cite{Der99}. That is, if $f_1, \ldots, f_r$ are generators of  $I_G$, then $\rho(\pi(f_1)), \ldots, \rho(\pi(f_r))$ generate $R^G$ as a $\KK$-algebra, where $\pi : S \rightarrow R$ is the projection and $\rho : R \rightarrow R^G$ is the Reynolds operator.\\

Given an ideal $I \subseteq R$ and $n \in \NN$, the $n$-th symbolic power of $I$ is defined as
$$
I^{(n)} = \bigcap_{\fp \in \Ass(R/I)}  I^n R_{\fp} \cap R
$$ 
where $\Ass(R/I)$ is the set of associated primes of $I$. Symbolic powers of ideals have been studied intensely over the last two decades (see \cite{DDSG+18} for a recent survey). Given that the ordinary power is always contained in the symbolic power, it is natural to ask if the other containment holds for every natural number. In general, this question has a negative answer. For instance, consider the prime ideal $\fp=(x,y)$ in the ring ${\large \sfrac{\KK[x,y,z]}{(x^n - yz)}}$. In this case, $y \in \fp^{(2)}$ but $y \not\in \fp^2$. 

Given the previous fact, another question arises: under what conditions does the equality hold? There are several cases in which the equality holds for every natural number \cite{Hoc73, HSV89}. We study the equality between the symbolic and the ordinary powers of the Derksen ideal.\\

Symbolic powers of ideals which arise from the action of a group have been of recent interest. They help to give counterexamples to the Harbourne's Conjecture; let $G$ be a finite group generated by pseudoreflections and determines an arrangement $\mathcal{A}$ of hyperplanes. Let $\mathcal{J}(\mathcal{A})$ be a radical ideal which defines the singular locus of the reflection arrangement $\mathcal{A}$. It is known that $\mathcal{J}(\mathcal{A})^{(3)} \not\subseteq \mathcal{J}(\mathcal{A})^2$ for certain groups $G$. Dumnicki, Szemberg, and Tutaj-Gasi\'nska \cite{DST13} and Harbourne and Seceleanu \cite{HS15} showed it for the infinite family of monomial groups $G(m,m,3)$; and Klein and Wiman showed it for two classical groups $G_{24}$ and $G_{27}$ respectively \cite{BDHHLPS, BDHHSS}. Drabkin and Seceleanu generalized this result for groups $G_{29}, G_{33}, G_{34}$ and $G(m,m,n)$ with $m, n \geq 3$ \cite{DS20}.\\

\hspace*{-0.6cm} In this manuscript we give a positive answer in several cases for the following question.

\vspace*{0.2cm}

\begin{questionx}[{Jeffries}]\label{ConjDerksen}
Let $R=\KK[x_1, \ldots, x_d]$ be any polynomial ring, $G$ be any finite group which acts linearly on $R$ and $I_G$ be the Derksen ideal. Is $I_G^n=I_G^{(n)}$ fulfilled for every $n\in\ZZ_{>0}$?
\end{questionx}

\vspace*{0.2cm}

\hspace*{-0.6cm} There is some computational evidence that this question has a positive answer. In this article, we prove the equality between symbolic and ordinary powers of Derksen ideals in the following cases.

\vspace*{0.2cm}

\begin{theoremx}[{see Theorem \ref{dimension 1}}]\label{dimension 1 intro}
Let $R = \KK[x]$ be a polynomial ring in one variable over the field $\KK$ and $G$ be any finite group acting faithfully on $R$. Then $I_G^{(n)} = I_G^n$ for every $n \geq 1$.
\end{theoremx} 

\vspace*{0.2cm}

\begin{theoremx}[{see Theorem \ref{Z2} and Theorem \ref{Z22}}]\label{Z2 intro}
Let $R = \KK[x_1, \ldots, x_d]$ be a polynomial ring over the field $\KK$ and $G=\ZZ/2\ZZ$. Then $I_G^{(n)} = I_G^n$ for every $n \geq 1$.
\end{theoremx}

\vspace*{0.2cm}

\begin{theoremx} [{see Theorem \ref{diagonal}}]\label{diagonal intro}
Let $R = \KK[x_1, \ldots, x_d]$ be a polynomial ring over the field $\KK$ and $G=\langle g_1, \ldots, g_a \rangle$ be a finite group with $a \leq d$ which acts linearly on $R$ in the following way:
$$
g_i(x_j)=
\begin{cases} 
     x_j & \hbox{if } j \neq i;\\
     \omega_{i}x_j & \hbox{if } j=i.
\end{cases}
$$
where $\omega_{i}$ is a $d_i$-root of unity with $d_i \geq 1$. Then, $I_G^{(n)} = I_G^n$ for every $n \geq 1$.
\end{theoremx}

\vspace*{0.2cm}

\hspace*{-0.55cm} In addition, we study the equality locally on the punctured spectrum.

\vspace*{0.2cm}

\begin{theoremx}[{see Theorem \ref{cyclic}}]\label{cyclic intro}
Let $R = \KK[x_1, \ldots, x_d]$ be a polynomial ring over a field $\KK$ and $G$ be a group in which each element $g \neq e$ fixes only the origin. Then, $I_G^{(n)} S_\fp = I_G^n S_\fp$ for all $n \in \NN$ and for all prime ideals $\fp $ different from $ (x_1, \ldots, x_d, y_1 , \ldots, y_d)$.
\end{theoremx}

\vspace*{0.2cm}

\section{Notation}\label{prelim}

\vspace*{0.2cm}

In this section we establish the notation used throughout the rest of the manuscript. We assume $R=\KK[x_1,\ldots, x_d]$ and $S = \KK[x_1, \ldots, x_d, y_1,  \ldots , y_d]$ are standard graded polynomial rings over the field $\KK$, $\fm=(x_1,\ldots,x_d)$ and $\fn = (x_1, \ldots, x_d, y_1, \ldots, y_d)$ are the homogeneous maximal ideals of $R$ and $S$ respectively, and $G$ is any finite group.

\vspace*{0.2cm}

\begin{definition} \label{def group action} 
Let $G$ be a group and $X$ be a set. A group action of $G$ on $X$ is a function $\varphi : G \times X \rightarrow X$ such that
\begin{itemize}
\item[i)] $\varphi(e,x) = x$ for every $x \in X$ where $e$ is the identity of $G$.

\item[ii)] $\varphi(gh,x) = \varphi(g , \varphi(h,x))$ for every $x \in X$ and every $g, h \in G$.
\end{itemize}
We denote $\varphi(g,x)$ by $g(x)$.
\end{definition}

\newpage

\hspace*{-0.6cm} We say that $G$ acts linearly on $R$ if $g(a) = a$ for every $a \in \KK$ and $g(x_i)$ is an homogeneous linear polynomial for every $i=1,\ldots, d$.\\

\hspace*{-0.56cm} We denote by $R^G$ the invariant ring, that is, the set of elements $f \in R$ such that $g(f) = f$ for every $g \in G$.\\

\hspace*{-0.6cm} If $G$ is linearly reductive, there exists an unique $G$-invariant linear projection $\rho : R \rightarrow R^G$, that is, $\rho(g \cdot f) = \rho(f)$ for all $g \in G$ and $f \in R$, and $\rho(f) = f$ for all $f \in R^G$, which is called the Reynolds operator \cite{Spr89}. It has the following properties:
\begin{itemize}
\item[1.] $\rho$ is a $R^G$-module homomorphism.
\item[2.] If $W \subseteq R$ is a $G$-invariant linear subspace, then $\rho(W) = W^G$.
\end{itemize}
A finite group $G$ is linearly reductive if and only if $|G|$ is a unit in $\KK$ \cite{Spr89}. In this case the Reynolds operator is
$$
\rho(f) = \dfrac{1}{|G|} \displaystyle\sum_{g \in G} g(f).
$$

\vspace*{0.4cm}

\section{Derksen Ideals} \label{Section 3}

\vspace*{0.2cm}

\hspace*{-0.56cm} In order to find generators to the invariant ring, it is enough to find homogeneous generators of the zero-fiber ideal. Corollary \ref{corol 3.5} motivates interest in the Derksen ideal. 

\vspace*{0.2cm}

\begin{definition} \label{def Derksen ideals}
For a linear action of a finite group $G$ on the polynomial ring $R$ we define the Derksen ideal as \index{Derksen ideal} the ideal
$$
I_G = \bigcap_{g \in G} (y_1 - g(x_1), \ldots, y_d - g(x_d))
$$
in $S$.
\end{definition}

\vspace*{0.1cm}

\begin{definition} \label{def zero-fiber}
We define the zero-fiber ideal as the ideal of $R$ generated by all homogeneous invariants of positive degree, and we denote it by $I_{\mathscr{N}}$.
\end{definition}

\vspace*{0.2cm}

\hspace*{-0.56cm} The notation of the zero-fiber is motivated by the fact that its zero set is called the nullcone which we denote by $\mathscr{N}$.\\

\hspace*{-0.56cm} We now recall several properties of Derksen ideals.

\vspace*{0.2cm}

\begin{prop}[\cite{Der99}] \label{prop generators of invariant ring}
Assume $G$ is linearly reductive. If $h_1, \ldots, h_s $ are homogeneous and generate $I_{\mathscr{N}}$, then $\rho(h_1), \ldots, \rho(h_s)$ generate $R^G$ as a $\KK$-algebra.
\end{prop}

\vspace*{0.2cm}

\begin{prop}[\hspace*{-0.15cm}{{{ \cite[Theorem 3.1]{Der99}}}}] \label{prop 3.4}
We have the equality
$$
((y_1, \ldots , y_d) + I_G ) \cap R = I_{\mathscr{N}}.
$$
\end{prop}

\vspace*{0.3cm}

\begin{cor}[\hspace*{-0.15cm} \cite{Der99}] \label{corol 3.5}
If the ideal $I_G$ is generated by $f_1(\underline{x}, \underline{y}), \ldots, f_r(\underline{x}, \underline{y})$, then $f_1(\underline{x}, 0), \ldots,$ $f_r(\underline{x}, 0)$ generate $I_{\mathscr{N}}$.
\end{cor}

\newpage

\begin{thm}[{{{Cf. }}}\cite{Duf09}] \label{theo equality Derksen ideal and AG}
If the field $\KK$ is algebraically closed, then the Derksen ideal $I_G$ is equal to $\sqrt{( \{f(\underline{x}) - f(\underline{y}) : f \in R^G \} )}$.
\end{thm}
\begin{proof}
Let $J \subseteq R$ be an ideal. We denote by $\VV(J) $ the vanishing of $J$. We observe that $I_G = \sqrt{( \{f(\underline{x}) - f(\underline{y}) : f \in R^G  \})}$ is equivalent to $\VV \left(\sqrt{( \{f(\underline{x}) - f(\underline{y}) : f \in R^G \})}\right) =  \VV(I_G) $, because $I_G$ is radical. We note that
$$
\VV( \{ f(\underline{x}) - f(\underline{y}) : f \in R^G \}) = \VV \left(\sqrt{( \{f(\underline{x}) - f(\underline{y}) : f \in R^G \})} \right) ,
$$
and
$$
\begin{array}{rcl} \vspace*{0.2cm}
\VV(I_G) &=& \VV \left( \bigcap_{g \in G} (y_1 - g(x_1), \ldots, y_d - g(x_d)) \right) \\\vspace*{0.2cm}
&=& \bigcup_{g \in G} \VV (\{ y_1 - g(x_1), \ldots, y_d - g(x_d) \}) \\\vspace*{0.2cm}
&=& \displaystyle\bigcup_{g \in G} \{ (\underline{a}, g(\underline{a})) \ | \ a \in \KK^d \}.
\end{array}
$$
Therefore, we need to show that $\VV( \{ f(\underline{x}) - f(\underline{y}) : f \in R^G \}) = \displaystyle\bigcup_{g \in G} \{ (\underline{a}, g(\underline{a})) \ | \ a \in \KK^d \}$. Let $(a, g(a)) \in \KK^{2d}$ for some $a \in \KK^d$ and some $g \in G$. We note that
$$
f(a) - f(g(a)) = f(a) - g(f(a)) = f(a) - f(a) = 0
$$
for any $f \in R^G$. Then, $\displaystyle\bigcup_{g \in G} \{ (\underline{a}, g(\underline{a})) \ | \ a \in \KK^d \} \subseteq \VV( \{ f(\underline{x}) - f(\underline{y}) : f \in R^G \})$. Conversely, let $(a,b) \in \KK^d \times \KK^d$ such that $b \not\in G \cdot a$. There exists $\ell(x) = \sum_{i=1}^n c_i x_i + c_0$ such that $\ell(b)=0$ and $\ell(g(a)) \neq 0$ for all $g \in G$. We define
$$
f = \prod_{g \in G } g \circ \ell .
$$
We observe that for each $\sigma \in G$, we have $\sigma(f)(x) = \displaystyle\prod_{g \in G} \sigma(g \circ \ell)(x) = \displaystyle\prod_{g \in G} (\sigma \circ g)(\ell(x)) = \displaystyle\prod_{h \in G} h (\ell(x)) = f(x)$. Therefore, $f \in R^G$. Furthermore,
$$
f(b) = \prod_{g \in G} g(\ell(b)) = \prod_{g \in G} g(0) = 0 \ \ \ \text{ and } \ \ \  f(a) = \prod_{g \in G} g(\ell(a)) = \ell(g(a)) \neq 0
$$
because $\ell(g(a))\neq 0$ for every $g \in G$. So, there is $f \in R^G$ such that $f(a) \neq f(b)$. 
\end{proof}

\vspace*{0.4cm}

\section{Symbolic Powers of Derksen Ideals}

\vspace*{0.2cm}

\hspace*{-0.56cm} In this section we show Theorems \ref{dimension 1 intro}, \ref{Z2 intro}, and \ref{diagonal intro}.  We start by giving a known description for symbolic powers of Derksen ideals. This result is useful in order to show equality between symbolic and ordinary powers.

\vspace*{0.1cm}

\begin{lemma} \label{intersection of symbolic powers}
Let $G$ be a finite group which acts linearly on the polynomial ring $R$. If $I_G$ is the Derksen ideal, then for all $n \geq 1$
$$
I_G^{(n)} = \bigcap_{g \in G} (y_1 - g(x_1), \ldots , y_d - g(x_d))^n.
$$
\end{lemma}
\begin{proof}
Since each ideal $(y_1 - g(x_1), \ldots , y_d - g(x_d))$ is generated by a regular sequence, we have $(y_1 - g(x_1), \ldots , y_d - g(x_d))^{(n)} = (y_1 - g(x_1), \ldots, y_d - g(x_d))^n$ \cite{Hoc73}. Hence,
$$
I_G^{(n)} = \displaystyle\bigcap_{g \in G} (y_1 - g(x_1), \ldots , y_d - g(x_d))^{(n)} = \displaystyle\bigcap_{g \in G} (y_1 - g(x_1), \ldots , y_d - g(x_d))^n. 
$$ 
\end{proof}

\vspace*{0.2cm}

\begin{thm} \label{dimension 1}
Let $R = \KK[x]$ be a polynomial ring in one variable and $G$ be any finite group acting faithfully on $R$. Then, $I_G^{(n)} = I_G^n$ for all $n \in \NN$.
\end{thm}
\begin{proof}
Since $G$ acts faithfully on $R$, $g(x) \neq h(x)$ for all $g, h \in G$. Therefore, the ideals $(y - g(x))$ for each $g \in G$ are distinct and so the ideals $(y-g(x))^n$ for each $g \in G$ are distinct. In addition, given $G$ acts linearly on $R$, $\{ y - g(x) \ | \ g \in G \}$ are irreducible polynomials. Thus,
$$
\begin{array}{rcl}
\displaystyle\bigcap_{g \in G} (y - g(x)) = \displaystyle\prod_{g \in G} (y - g(x)) \ \ \ \ \ \ \text{ and } \ \ \ \ \ \ \displaystyle\bigcap_{g \in G} (y - g(x))^n = \displaystyle\prod_{g \in G} (y - g(x))^n.
\end{array}
$$
Therefore,
$$
\begin{array}{rcl} \vspace*{0.3cm}
I_G^{(n)} &=& \displaystyle\bigcap_{g \in G} (y - g(x))^n =  \displaystyle\prod_{g \in G} (y - g(x))^n  \\
&=& \left( \displaystyle\prod_{g \in G} (y - g(x)) \right)^n  = \left( \displaystyle\bigcap_{g \in G} (y - g(x)) \right)^n = I_G^{n}.   
\end{array} 
$$
\end{proof}


\vspace*{0.6cm}

\begin{example}
Let $R = \KK[x_1,x_2]$ and $G = \{1, g \}$ where
$$
g \begin{pmatrix}
x_1 \\
x_2 
\end{pmatrix} = \begin{pmatrix}
-x_1 \\
x_2 
\end{pmatrix}.
$$
Then
$$
\begin{array}{rcl} \vspace*{0.3cm}
I_G = (y_1 - x_1, y_2 - x_2) \cap (y_1 + x_1, y_2 - x_2 ) = (y_1^2 - x_1^2, y_2 - x_2)
\end{array}
$$
Note that $\{y_1^2 - x_1^2, y_2 - x_2 \}$ is a regular sequence. Thus, $I_G^{(n)} = I_G^n$ for every $n \geq \NN$.\\

\hspace*{-0.6cm} More generally, for every group of order 2 and any polynomial ring, the symbolic powers of the Derksen Ideal are equal to the ordinary powers as we show next.
\end{example}

\vspace*{0.5cm}

\hspace*{-0.6cm} For Theorem \ref{Z2 intro} we separate the proof into two cases: characteristic $2$ and different than $2$. The difference is based in the associated matrix to the action of the element different from the identity in $G$. This matrix is diagonalizable when $\Char(\KK) \neq 2$, while this is not necessarily true when $\Char(\KK)=2$. However for the second case, its Jordan Canonical Form matrix is simple enough to yield a similar proof of the diagonalizable case. To show Theorem \ref{Z2 intro}, we first show the characterization of these matrices.\\

\hspace*{-0.65cm} Let $R = \KK[x_1, \ldots, x_d]$ be a polynomial ring over a field $\KK$ and a group $G = \{1, g \}$. Let $A = (a_{ij}) \in \Mat_{d \times d}(\KK)$ be the associated matrix of the action of $g$ on $R$, that is, 
$$
\begin{pmatrix}
a_{11} & a_{12} & \cdots & a_{1d} \\
a_{21} & a_{22} & \cdots & a_{2d} \\
\vdots & \vdots &  & \vdots \\
a_{d1} & a_{d2} & \cdots & a_{dd}
\end{pmatrix} \begin{pmatrix}
x_1 \\
x_2 \\
\vdots \\
x_d
\end{pmatrix} = \begin{pmatrix}
g(x_1) \\
g(x_2) \\
\vdots \\
g(x_d)
\end{pmatrix}
$$
Since $g^2=1$, we have that $A^2=I$. Thus, the eigenvalues of $A$ are $\{1, -1 \} \subseteq \KK$. Hence, $A$ is similar to a Jordan matrix
$$
\begin{pmatrix}
J_1 & & \\
& \ddots & \\
& & J_n
\end{pmatrix} \hspace*{0.5cm} \text{ where } \hspace*{0.5cm} J_i = \begin{pmatrix}
\lambda_i & 1 & 0 & \cdots \\
 & \lambda_i & 1 & \cdots \\
 & & \ddots \\ 
 &  & & \lambda_i
\end{pmatrix}
$$
and $\lambda_i \in \{1, -1\}$. Note that
$$
A^2 = \begin{pmatrix}
J_1^2 & & \\
& \ddots & \\
& & J_n^2
\end{pmatrix} \hspace{0.5cm} \text{and} \hspace*{0.5cm} J_i^2 = \begin{pmatrix}
\lambda_i^2 & 2 \lambda_i & 1 & \cdots \\
 & \lambda_i^2 & 2\lambda_i & \cdots \\
 & & \ddots \\ 
 &  & & \lambda_i^2
\end{pmatrix}.
$$
Given that $A^2 = I$, we have $J_i^2 = I$ for each $i$. Therefore, if $\Char(\KK) \neq 2$, then $\lambda_i = \pm 1$ and $J_i \in \Mat_{1 \times 1}(\KK)$. On the other hand, if $\Char(\KK) = 2$, then $\lambda_i=1$ and $J_i \in \Mat_{1 \times 1}(\KK)$ or $J_i \in \Mat_{2 \times 2}(\KK)$.\\

We assume that $\Char(\KK) \neq 2$. In this case, each $J_i \in \Mat_{1 \times 1}(\KK)$ and so $A$ is similar to a diagonal matrix
$$
\begin{pmatrix}
\lambda_1 & \\
& \lambda_2 \\
& & \ddots \\
& & & \lambda_d
\end{pmatrix}
$$
where $\lambda_i = \pm 1$. Rearranging if it is necessary, we choose $\lambda_i=1$ for $i=1, \ldots, j-1$ and $\lambda_i = -1$ for $i=j, \ldots, d$ for some $1 \leq j \leq d$. Hence, the action of $g$ is given by 
$$
\begin{array}{lcl} \vspace*{0.2cm}
g(x_i) = x_i & \text{ for all } & 1 \leq i \leq j-1 \\
g(x_i) = -x_i & \text{ for all } & j \leq i \leq d \\
\end{array}
$$

\newpage

We now assume that $\Char(\KK) = 2$. In this case, some $J_i \in \Mat_{1 \times 1}(\KK)$ and other are in $\Mat_{2 \times 2}(\KK)$. Thus, rearranging if it is necessary, $A$ is similar to a matrix of the form
\[ A \sim
\left(\begin{array}{@{}c|c@{}}
  \begin{matrix}
  1 & 0 & \cdots & 0 \\
  0 & 1 & \cdots & 0 \\
  \vdots & \vdots & \ddots & \vdots \\
  0 & 0 & \cdots & 1
  \end{matrix}
  & \bigzero \\
\hline
  \bigzero & \begin{array}{@{}c|c@{}}
  \begin{matrix}
  1 & 1 \\
  0 & 1
  \end{matrix}
  & \bigzero \\
\hline
  \bigzero & \begin{array}{@{}c|c@{}}
  		\begin{matrix}
  			1 & 1 \\
  			0 & 1
  		\end{matrix}
  		& \bigzero \\
		\hline
  		\bigzero & \begin{array}{@{}c|c@{}}
  				\ddots
  				& \bigzero \\
				\hline
  				\bigzero & 
  				\begin{matrix}
  					1 & 1 \\
  					0 & 1
  				\end{matrix}
  		\end{array}
  \end{array}
  \end{array}
\end{array}\right)
\]
Hence, the action of $g$ is given by 
$$
\begin{array}{lcl} \vspace*{0.2cm}
g(x_i) = x_i & \text{ for all } & 1 \leq i \leq j-1 \\\vspace*{0.2cm}
g(x_i) = x_i + x_{i+1} & \text{ for all } & i = j, j+2, \ldots, d-1 \\\vspace*{0.2cm}
g(x_i) = x_i & \text{ for all } & i = j+1, j+3, \ldots, d \\\vspace*{0.2cm}
\end{array}
$$

\begin{thm} \label{Z2}
Let $R = \KK[x_1, \ldots, x_d]$ be a polynomial ring over a field $\KK$ of characteristic different from $2$ and $G = \ZZ / 2\ZZ$. Then, $I_G^{(n)} = I_G^n$ for all $n \in \NN$.
\end{thm}
\begin{proof}
Let $G = \{1, g \}$. By the characterization in the $\Char(\KK) \neq 2$ case, we have that $g(x_i) = x_i$ for all $i < j$ and $g(x_i) = -x_i$ for all $i \geq j$ where $1 \leq j \leq d$. Thus,
$$
I_G = (y_1 - x_1, \ldots , y_d - x_d ) \cap (y_1 - x_1, \ldots , y_{j-1} - x_{j-1}, y_j + x_j, \ldots , y_d + x_d ).
$$
Consider the polynomial rings $T = \KK[a_1, \ldots, a_d, b_1, \ldots, b_d]$ and $S = \KK[x_1, \ldots, x_d, y_1, \ldots, y_d]$ over the field $\KK$. We define the map $\varphi : T \rightarrow S $ as $\varphi(a_i) = y_i - x_i$ and $\varphi(b_i) = y_i + x_i$ for $i = 1, \ldots, d$. Likewise, we define the map $\phi : S \rightarrow T$ by $\phi(x_i) = \frac{b_i - a_i}{2}$ and $\phi(y_i) = \frac{b_i + a_i}{2}$ for $i=1, \ldots, d$. We observe that $\varphi$ and $\phi$ are inverse morphisms. Hence, $S$ and $T$ are isomorphic, which implies that $I_G = \varphi(J)$ where 
$$
\begin{array}{rcl} \vspace*{0.2cm}
J &=& (a_1, \ldots , a_d) \cap (a_1, \ldots , a_{j-1}, b_j, \ldots, b_d) \\\vspace*{0.2cm}
&=& (a_1, \ldots , a_{j-1}, a_jb_j, \ldots, a_j b_d, \ldots , a_db_j, \ldots , a_db_d) \\\vspace*{0.2cm}
&=& (a_1, \ldots , a_{j-1}) + (a_jb_j, \ldots, a_j b_d, \ldots , a_db_j, \ldots , a_db_d).
\end{array}
$$
Denote $J_1 = (a_1, \ldots , a_{j-1})$ and $J_2 = (a_jb_j, \ldots, a_j b_d, \ldots , a_db_j, \ldots , a_db_d)$. We claim that $J^{(n)} = J^n$ for every $n$. By Theorem 7.8 in \cite{BCGHJNSVV16} we have that
\begin{equation} \label{eq 3}
J^{(n)} = \displaystyle\sum_{k=0}^n J_1^{(n-k)} J_2^{(k)} = \displaystyle\sum_{k=0}^n J_1^{n-k} J_2^{(k)}
\end{equation}
because $J_1$ is generated by a reqular sequence. Given that
\begin{equation} \label{eq 4}
J^n =  \displaystyle\sum_{k=0}^n J_1^{n-k} J_2^{k}
\end{equation}
the (\ref{eq 3}) and $(\ref{eq 4})$ implies that $J^{(n)} = J^n$ if and only if $J_2^{(k)} = J_2^k$ for every $k \in \NN$.\\

\hspace*{-0.6cm} Let $G$ be a graph with vertices $\{a_j, \ldots, a_d, b_j, \ldots, b_d \}$ and edges $\{ \{a_l, b_m \} \ : \ {j \leq l, m \leq d} \}$. Observe that $G$ is a bipartite graph with edge ideal
$$
( \{ a_l b_m \ : \ j \leq l, m \leq d \} ) = (a_j b_j, \ldots, a_j b_d, \ldots, a_d b_j, \ldots, a_d b_d) = J_2.
$$
and so $J_2^{(k)} = J_2^k$ for all $k \in \NN$ because the symbolic powers and usual powers of an edge ideal of a bipartite graph coincide for any $k \in \NN$ \cite{GVV05}. Therefore, $J^{(n)} = J^n$. Hence, 
$$
I_G^{(n)} = \varphi(J)^{(n)} = \varphi(J^{(n)}) = \varphi(J^n) = \varphi(J)^n = I_G^n
$$
because $\varphi$ is an isomorphism.
\end{proof}

\vspace*{0.3cm}

\begin{thm} \label{Z22}
Let $R = \KK[x_1, \ldots, x_d]$ be a polynomial ring over a field $\KK$ of characteristic $2$ and $G = \ZZ / 2\ZZ$. Then, $I_G^{(n)} = I_G^n$ for all $n \in \NN$.
\end{thm}
\begin{proof}
Let $G = \{1, g \}$. By the characterization in $\Char(\KK)=2$ case, we have that $g(x_i) = x_i$ for all $i = 1, \ldots, j-1, j+1, j+3, \ldots, d$ and $g(x_i) = x_i + x_{i+1}$ for all $i = j, j+2, \ldots, d-1$ for some $j \in \{1, \ldots, d\}$. Thus, 
$$
\begin{array}{rcl} \vspace*{0.2cm}
I_G &=& (y_1 - x_1, \ldots , y_d - x_d ) \cap (y_1 - x_1, \ldots , y_{j-1} - x_{j-1}, y_j - x_j - x_{j+1}, y_{j+1} - x_{j+1}, \\\vspace*{0.2cm}
& & \hspace*{3cm} y_{j+2} - x_{j+2} - x_{j+3}, y_{j+3} - x_{j+3}, \ldots , y_{d-1} - x_{d-1} - x_d, y_d - x_d ).
\end{array}
$$
Consider the variables $a_i = y_i - x_i$ for $i=1, \ldots, j-1$, $a'_i = y_i - x_i$ for $i=j, j+1, \ldots, d$, and $b'_i = y_i - x_i - x_{i+1}$ for $i = 1, \ldots, d-1$. Observe that this variables are algebraically independent. Thus, we can extend it to a basis of $R$. Therefore, we can write the Derksen ideal in this basis as  
$$
\begin{array}{rcl} \vspace*{0.3cm}
I_G &=& (a_1, \ldots , a_{j-1}, a'_{j}, \ldots, a'_d) \cap (a_1, \ldots , a_{j-1}, b'_j, a'_{j+1}, b'_{j+2}, a'_{j+3}, \ldots , b'_{d-1}, a'_d) 
\end{array}
$$
We rename the variables as follows
$$
\begin{array}{rclcrclcrcl} \vspace*{0.2cm}
a_j &=& a'_{j+1} & \hspace*{1cm} & a_{l+1} &=& a'_{j} & \hspace*{1cm} & b_{l+1} &=& b'_{j} \\
a_{j+1} &=& a'_{j+3} & &  a_{l+2} &=& a'_{j+2} & & b_{l+2} &=& b'_{j+2} \\
& \vdots & & & & \vdots & & & & \vdots & \\\vspace*{0.2cm}
a_l &=& a'_{d} & & a_d &=& a'_{d-1} & & b_d &=& b'_{d-1}
\end{array}
$$
Hence, we write $I_G$ as
$$
\begin{array}{rcl}
I_G &=& (a_1, \ldots , a_d) \cap (a_1, \ldots , a_l, b_{l+1}, \ldots, b_d).
\end{array}
$$
Notice that the ideal $I_G$ is the same as ideal $J$ in the previous proof. Therefore, $I_G^{(n)} = I_G^n$.
\end{proof}

\vspace*{0.2cm}

\begin{thm} \label{diagonal}
Let $R = \KK[x_1, \ldots, x_d]$ be a polynomial ring over an algebraically field $\KK$ and $G = \langle g_1, \ldots, g_a \rangle$ be a finite group with $a \leq d$ which acts linearly in $R$ of the following way
$$
g_i(x_j) = 
  \begin{cases}
    x_j   & \quad \text{if } j \neq i \\
    \omega_{i}x_j    & \quad \text{if } j=i , \\
  \end{cases}
$$
where $\omega_{i}$ is a $d_i$-root of unity, that is $\omega_{i}^{d_i} = 1$ with $d_i \geq 1$. Then, $I_G^{(n)} = I_G^n$ for every $n \in \NN$.
\end{thm}
\begin{proof}
Note that $x_1^{d_1}, \ldots, x_a^{d_a}, x_{a+1}, ..., x_d \in R^G$, because
$$
g_i(x_j^{d_j}) = (g_i(x_j))^{d_j} = x_j^{d_j}
$$
for $j=1, \ldots, a$ with $i \neq j$, $g_i(x_j) = x_j$ for $j=a+1, \ldots, d$, and 
$$
g_i(x_i^{d_i}) = (g_i(x_i))^{d_i} = (\omega_{i} x_i)^{d_i} = \omega_{i}^{d_i} x_i^{d_i} = x_i^{d_i}. \vspace*{0.2cm}
$$
We claim that $R^G = \KK[x_1^{d_1}, x_2^{d_2}, \ldots, x_a^{d_a}, x_{a+1}, \ldots, x_d]$. Given that the action of $G$ is linear, $g_i(f)$ does not modify the exponents of the variables on any polynomial $f \in R$, it only multiply a scalar to each monomial of $f$. Thus, it is enough to consider $f \in R^G$ a monomial in order to show the equality. Let $f = x_1^{\beta_1} x_2^{\beta_2} \cdots x_a^{\beta_a} x_{a+1}^{\beta_{a+1}} \cdots x_d^{\beta_d} \in R^G$ with $\beta_i \in \NN$. For each $i=1, \ldots, a$ we have
$$
g_i(f) = x_1^{\beta_1} x_2^{\beta_2} \cdots \omega_{i}^{\beta_i} x_i^{\beta_i} \cdots  x_{a}^{\beta_a} x_{a+1}^{\beta_{a+1}} \cdots x_d^{\beta_d}.
$$
Since $g_i(f)=f$, we have that $\omega_i^{\beta_i}=1$ and so $\beta_i$ is divisible by $d_i$ for each $i=1, \ldots, a$. Thus, $R^G = \KK[x_1^{d_1}, x_2^{d_2}, \ldots, x_a^{d_a}, x_{a+1}, \ldots, x_d]$. Therefore, by Theorem (\ref{theo equality Derksen ideal and AG}) we obtain that
$$
\begin{array}{rcl} \vspace*{0.3cm}
I_G &=& \sqrt{(\{f(y) - f(x) \ | \ f \in R^G\})} \\
&=& \sqrt{(y_1^{d_1} - x_1^{d_1}, y_2^{d_2} - x_2^{d_2}, \ldots, y_a^{d_a} - x_a^{d_a}, y_{a+1} - x_{a+1}, \ldots, y_d - x_d)}.
\end{array}
$$
Denote $(y_1^{d_1} - x_1^{d_1}, y_2^{d_2} - x_2^{d_2}, \ldots, y_a^{d_a} - x_a^{d_a}, y_{a+1} - x_{a+1}, \ldots, y_d - x_d)$ by $J_a$. We claim that $J_a$ is radical for any $a \in \NN$. We proceed by induction on $a$. If $a = 0$, then ${\large \sfrac{S}{J_0}} \cong \KK[x_1, \ldots, x_d]$ which is a domain. This implies that $J_0$ is prime, so it is radical.\\
We assume that $a = 1$. Then,
$$
{\large \sfrac{S}{J_1} \cong \sfrac{\KK[x_2, \ldots, x_d][x_1, y_1]}{(y_1^{d_1} - x_1^{d_1})}.}
$$
Note that $y_1^{d_1} - x_1^{d_1}$ is equal to the product of irreducible polynomials of multiplicity one. This implies that $(y_1^{d_1} - x_1^{d_1})$ is radical in $\KK[x_1, \ldots, x_d, y_1]$. Then, ${\large \sfrac{\KK[x_1, \ldots, x_d, y_1]}{(y_1^{d_1} - x_1^{d_1})}}$ is reduced and so ${\large \sfrac{S}{J_1}}$. Therefore, $J_1$ is radical.\\

\hspace*{-0.65cm} We now assume that $J_{a-1}$ is radical. Note that \\
$$
\begin{array}{rcl} \vspace*{0.3cm}
\sfrac{S}{J_a} & \cong & {\large\sfrac{\KK[x_1, \ldots, x_d, y_1 , \ldots , y_a]}{(y_1^{d_1} - x_1^{d_1}, \ldots , y_{a-1}^{d_{a-1}} - x_{a-1}^{d_{a-1}}, y_a^{d_a} - x_a^{d_a})} }\\\vspace*{0.3cm}
& \cong & {\large \sfrac{ \left(\sfrac{\KK[x_1, \ldots, x_{a-1}, x_{a+1}, \ldots, x_d, y_1, \ldots, y_{a-1}]}{(y_1^{d_1} - x_1^{d_1}, ..., y_{a-1}^{d_{a-1}} - x_{a-1}^{d_{a-1}})} \right) [x_a, y_a]}{ (y_a^{d_a} - x_a^{d_a})}}
\end{array}
$$
By induction hypothesis $\sfrac{\KK[x_1, \ldots, x_{a-1}, x_{a+1}, \ldots, x_d, y_1, \ldots, y_{a-1}]}{(y_1^{d_1} - x_1^{d_1}, ..., y_{a-1}^{d_{a-1}} - x_{a-1}^{d_{a-1}})}$ is reduced. Then, by case $a=1$, we conclude that $\sfrac{\KK[x_1, \ldots, x_d, y_1 , \ldots y_a]}{(y_1^{d_1} - x_1^{d_1}, \ldots , y_{a-1}^{d_{a-1}} - x_{a-1}^{d_{a-1}}, y_a^{d_a} - x_a^{d_a})}$ is reduced. Therefore, $J_a$ is radical.\\

\hspace*{-0.65cm} Hence, $I_G = J_a$. Since $J_a$ is generated by a regular sequence, we obtain that $I_G^{(n)} = I_G^n$ for every $n \in \NN$.
\end{proof}

\newpage

\section{Symbolic Powers of Derksen Ideals Locally}

\vspace*{0.2cm}

\hspace*{-0.6cm} We assume that $\KK$ is algebraically closed and we denote $V = \KK^d$. We define the separating variety as $S_{V,G} = \{(u, v) \in V \times V \ | \ f(u) = f(v) \ \ \text{for all } f \in R^G \}$. By Theorem \ref{theo equality Derksen ideal and AG} we have that $I_G = \II(S_{V,G})$. We denote $J_g = (y_1 - g(x_1), \ldots, y_d - g(x_d))$ for each $g \in G$ and $L_g = \VV(J_g) = (1 \otimes g)(V) = \{(u,g(u)) \ | \ u \in V \}$.

\vspace*{0.2cm}

\begin{lemma}[{{{\cite[Lemma 2.2]{DJ15}}}}]
Let $G$ be a finite group acting linearly on $R$. If $g,h \in G$, then 
$$
(1 \otimes g)(V) \cap (1 \otimes h)(V) = (1 \otimes h)(V^{h^{-1}g}) ,
$$
where $V^{h^{-1}g} =  \{ u \in V  \ | \ h^{-1}g(u) = u \}$.
\end{lemma}

\vspace*{0.2cm}

\begin{prop}\label{characterization}
Let $G$ be a nontrivial finite group which acts linearly on $R = \KK[x_1, \ldots, x_d]$. Then every element $g \neq e$ fixes only the origin if and only if $\Sing(S/I_G) = \{ \fn \}$.
\end{prop}
\begin{proof}
We assume that each $g \neq e$ fixes only the origin. This is equivalent to the fixed set of $h^{-1}g$, $V^{h^{-1}g}$ to be $\{0\}$ for all $h \neq g$. Therefore, 
$$
\VV(J_g + J_h) = \VV(J_g) \cap \VV(J_h) = (1 \otimes g)(V) \cap (1 \otimes h)(V) = (1 \otimes h)(V^{h^{-1}g}) = (1 \otimes h)(\{0\}) = \{0\}.
$$
Hence $\VV(J_g + J_h) =  \{0\} = \VV(\fn)$. Thus, $J_g + J_h = \fn$, because $J_g + J_h$ is a radical ideal.\\
Let $\fp$ be a prime ideal different from $\fn$. If $J_g \subseteq \fp$, then $\fp$ does not contain $J_h$ for any $h \neq g$; otherwise, $\fp = \fn$. Therefore, $I_GS_\fp = J_gS_\fp$ which is generated by variables. Hence $\fp \not\in \Sing(S/I_G)$. If $\fp$ does not contain any $J_g$, then $I_GS_\fp = S_\fp$ and so $\fp \not\in \Sing(S/I_G)$. Thus, $\Sing(S/I_G) \subseteq \{\fn\}$. If $\fn \not\in \Sing(S/I_G)$ we have that $S/I_G$ is a regular ring which implies that $G$ is trivial. Then, $\fn \in \Sing(S/I_G)$ and so $\Sing(S/I_G) = \{ \fn \}$.

\hspace*{-0.6cm} Conversely, we assume that $\Sing(S/I_G) = \{ \fn \}$. We claim that $\VV(J_g + J_h) \subseteq \Sing(S/I_G)$. Let $\fp$ be a prime ideal such that $J_g \subseteq \fp$ and $J_h \subseteq \fp$. We have that
$$
\begin{array}{rcl}
\Spec((S/I_G)_\fp) &\cong & \Spec(S/I_G) \cap \{ \text{primes that are contained in } \fp \} \\
&\cong & \VV(I_G) \cap \{ \text{primes that are contained in } \fp \} .
\end{array}
$$
We already know that $J_g, J_h \in \VV(I_G)$ and by hypothesis they are contained in $\fp$. In addition, they are minimal primes different from $0$. Therefore, $(S/I_G)_\fp$ is not a domain which implies it is not regular, because every local regular ring is a domain. Therefore, $\fp \in \Sing(S/I_G)$ and so the claim is done. Therefore, $\VV(J_g + J_h ) = \{ \fn \}$. \\
Hence, $\{0 \} = \{\fn\} = \VV(J_g + J_h) = \VV(J_g) \cap \VV(J_h) = L_g \cap L_h = (1 \otimes h)(V^{h^{-1}g})$ which implies that $h^{-1}g(u) = u$ only for $u=0$ for all $g \neq h$. Therefore, each $g \neq e$ fixes only the origin.
\end{proof}

\vspace*{0.2cm}

\begin{thm} \label{equality characterization}
Let $R = \KK[x_1, \ldots, x_d]$ be a polynomial ring over a field $\KK$ and $G$ be a group in which each element $g \neq e$ fixes only the origin. Then $I_G^{(n)} S_\fp = I_G^n S_\fp$ for all $n \in \NN$ and for all prime ideals $\fp$ different from $\mathfrak{n}$.
\end{thm}
\begin{proof}
By Proposition \ref{characterization}, we have that $\Sing(S/I_G) = \{ \fn \}$ and so $\dim(\Sing(S/I_G)) = 0$. Let $\A(I_G)$ be the union of the associated primes of $I_G^n$ for all $n \geq 1$. Let $\fp$ be a prime ideal such that $I_G \subseteq \fp$ and such that $\dim(S/\fp) > \dim(\Sing(S/I_G))=0$. This implies that $\fp \not\in \Sing(S/I_G)$, that is, $(S/I_G)_\fp$ is a regular local ring. Thus, $I_GS_\fp$ is a complete intersection. Therefore, $I_G^nS_\fp$ is an unmixed ideal. This means that $I_G^nS_\fp$ has no embedded primes, and so, $I_G^{(n)}S_\fp = I_G^nS_\fp$ for all $n$.
\end{proof}

\vspace*{0.2cm}

\hspace*{-0.6cm} By the previous proof we have that for every $\fp \not\in \Sing(S/I_G)$, $\fp$ is not an embedded prime of $I_G^n$ for all $n \geq 1$. Hence, $\A(I_G) \backslash \Min(I_G) \subseteq \Sing(S/I_G)$. 

\vspace*{0.4cm}

\begin{cor} \label{cyclic}
Let $R = \KK[x_1, \ldots, x_d]$ be a polynomial ring over a field $\KK$ and $G = \langle g \rangle$ be a cyclic group of order $t$ where $g(x_i) = \omega x_i$ and $\omega$ is a $t$-th root of unity which belongs to $\KK$. Then, $I_G^{(n)} S_\fp = I_G^n S_\fp$ for all $n \in \NN$ and for all prime ideals $\fp$ different from $\mathfrak{n}$.
\end{cor}
\begin{proof}
We observe that the action of $G$ in $V = \KK^d$ is defined by $g(a_1, \ldots, a_d) = (\omega a_1, \ldots, \omega a_d)$. Let $a = (a_1, \ldots, a_d) \in \KK^d$ be a fixed point by $g^m$ for some $m = 1, \ldots, t-1$. Then $(a_1, \ldots, a_d) = g^m(a_1, \ldots, a_d) = (\omega^m a_1, \ldots, \omega^m a_d)$ and so $a_i = \omega^m a_i$. If $a_i \neq 0$ then $\omega^m = 1$ which is not possible. Hence $a_i = 0 $ for all $i=1, \ldots, d$. Therefore, $g^m$ fixes only the origin for all $m=1, \ldots, t-1$. By Theorem \ref{equality characterization}, we obtain $I_G^{(n)}S_\fp = I_G^nS_\fp$ for all $\fp \neq \fn$.
\end{proof}

\section*{Acknowledgments}

\hspace*{-0.55cm} I am grateful to Jack Jeffries for suggesting the problem to me and for his helpful suggestions and comments. I thank Luis Núñez-Betancourt, Eloisa Grifo, and Alexandra Seceleanu for their useful comments. I acknowledge the support of Cátedras Marcos Moshinsky of Nuñez-Betancourt and CONACyT Grant 284598. Finally, I thank the Center for Research in Mathematics (CIMAT) and the University of Guanajuato where the research started. \\


\newcommand{\etalchar}[1]{$^{#1}$}
 \newcommand{\noop}[1]{}

\end{document}